\documentclass[11pt]{article}
\usepackage{amssymb,amsmath,amsthm,amsfonts}
\usepackage{mathrsfs}
\usepackage{subfigure}
\usepackage{graphicx}
\usepackage{epsfig}
\usepackage{float}
\usepackage[usenames]{color}

 \def\p{\partial} \def\nb{\nonumber}
\def\Vh0{\stackrel{\circ}{V}_h} \def\to{\rightarrow}
   
\def\Om{\Omega}   
\newcommand{\q}{\quad}
 \def\R{{\mathbb R}}
\def\l{\label}  \def\f{\frac}  
\def\qda{q^\delta_\beta}
\def\qd{q^\dag}

\def\A{\mathcal{A}}

\def\Jdb
\def\D{\end{document}}   
   
\def\m{\mbox}

\newcommand{\lc}
{\mathrel{\raise2pt\hbox{${\mathop<\limits_{\raise1pt\hbox
{\mbox{$\sim$}}}}$}}}

\newcommand{\gc}
{\mathrel{\raise2pt\hbox{${\mathop>\limits_{\raise1pt\hbox{\mbox{$\sim$}}}}$}}}

\newcommand{\ec}
{\mathrel{\raise2pt\hbox{${\mathop=\limits_{\raise1pt\hbox{\mbox{$\sim$}}}}$}}}

\def\bb{\begin{equation}}  \def\ee{\end{equation}}

\def\beqn{\begin{eqnarray}}  \def\eqn{\end{eqnarray}}

\def\beqnx{\begin{eqnarray*}} \def\eqnx{\end{eqnarray*}}

\def\bn{\begin{enumerate}} \def\en{\end{enumerate}}

\def\bd{\begin{description}} \def\ed{\end{description}}

\newtheorem{lemma}{Lemma}[section]
\newtheorem{theorem}{Theorem}[section]
\newtheorem{corollary}{Corollary}[section]
\newtheorem{remark}{Remark}[section]

 \def\x{{\bf x}}

\title{Convergence Rates of Tikhonov
Regularizations for \\Elliptic and Parabolic Inverse Radiativity  Problems }
\author{
Dehan Chen \footnote{School of Mathematics and Statistics $\&$ Hubei Key Laboratory of Mathematical
Sciences, Central  China Normal University, Wuhan, 430079, P.R.China. The
work of this author was financially supported by
National Natural Science Foundation of China
(Nos.  11701205 and 11871240 ). (dehan.chen@uni-due.de)}
\and
Daijun Jiang\footnote{School of Mathematics and Statistics $\&$ Hubei Key Laboratory of Mathematical
Sciences, Central  China Normal University, Wuhan, 430079, P.R.China. The
work of this author was financially supported by
National Natural Science Foundation of China
(Nos. 11871240, 11401241 and 11571265) and NSFC-RGC (China-Hong Kong, No. 11661161017).
(jiangdaijun@mail.ccnu.edu.cn)}
\and Jun
Zou\footnote{Department of Mathematics, The Chinese University of
Hong Kong, Shatin, Hong Kong. The work
of this author was substantially supported by Hong Kong RGC grants
(Project  14322516).  (zou@math.cuhk.edu.hk)}}

\begin{document}
\maketitle

\begin{abstract}
We shall study in this paper the  convergence rates of  Tikhonov
regularization for the recovery of the radiativities in elliptic and parabolic systems with
Dirichlet boundary conditions in general dimensional spaces. The conditional stability estimates are first
derived. Due to the difficulty of the verification of the existing source conditions or nonlinearity conditions for the
considered inverse radiativity  problems in high dimensional spaces, some new variational source
conditions are proposed. The conditions are rigorously verified in general dimensional spaces under the
conditional stability estimates. We shall finally derive the reasonable convergence rates, which
explicitly reveals the relation between the regularity of the radiativities and the convergence rates results.
\end{abstract}


\medskip
\section{Introduction}\label{sec:intro}
\setcounter{equation}{0}
Identification of radiativities can find wide applications in industry, physics and engineering
\cite{bank89} \cite{engl96} \cite{engl95} \cite{isakov98}.
The stationary diffusivity and radiativity problem is often modelled by
the elliptic boundary value problem
\begin{equation}
\left\{ \begin{array}{rlllc}
-\nabla\cdot(a(\x)\nabla u) +q(\x)u&=&f(\x) &\m{in} &\Om, \\
 u&=&g(\x) &\m{on}  &\p\Om,
\end{array}
\right. \label{q1}
\end{equation}
while the time-dependent diffusion and radiation process can be modelled by
the parabolic system
\begin{equation}
\left\{ \begin{array}{rlclc}
\p_t u-\nabla\cdot(a(\x)\nabla u)+q(\x)u&=&f(\x,t) &\m{in} &\Om\times (0,T], \\
 u(\x,0)&=&u_0(\x) &\m{in}  &\Om,\\
 u(\x,t)&=&g(\x,t) &\m{on}  &\p\Om\times (0,T],
\end{array}
\right. \label{q1p}
\end{equation}
where $\Om\subset \R^d~(d=2,3)$ is the interested physical domain,
an open bounded and connected domain with $C^2$ boundary $\p\Om$.
The source density $f(\x)$ or $f(\x,t)$, ambient temperature $g(\x)$ or $g(\x,t)$,
conductivity $a(\x)$ and the initial temperature $u_0(\x)$ are given, while
the radiativity $q(\x)$ is the focus of our interest to be
reconstructed  in the following admissible constraint set
\beqn
K=\Big\{q\in L^2(\Om);\,\,0<\underline q\leq q\leq\bar q \,\,{\rm a.e.\,\, in}\,\, \Om\Big\}.
 \l{eq:contraint}
\eqn
Here $\underline q$ and $\bar q$ are two positive constants. For convenience, we often write the
solutions of systems (\ref{q1}) and (\ref{q1p}) as $u(q)$ to emphasize their dependence on the
radiativities $q(\x)$.

We shall consider the following elliptic and parabolic inverse radiativity problems:

{\bf Elliptic Inverse Radiativity Problem }. Let $a(\x),f(\x)$ and $g(\x)$ be know in  (\ref{q1}),
recover the radiativity  $q(\x)$
in  $\Omega$ from the available noisy data $\nabla z^\delta$ (or $z^\delta$) of $\nabla u$  (or $u$) in $\Om$, where
$\delta$ is the noise level.

{\bf Inverse Parabolic Radiativity Problem}. Let $a(\x),f(\x,t)$, $g(\x,t)$ and $u_0(\x)$ be know in  (\ref{q1p}),
identify the radiativity  $q(\x)$ in $\Omega$  from the available noisy data
$\nabla z^\delta$ (or $z^\delta$) of $\nabla u$  (or $u$) in $\Om\times I$, where
$I$ is an open subinterval of $(0,T]$.

%

Convergence rates have been well studied for Tikhonov regularizations
for inverse conductivity and radiativity problems \cite{engl89} \cite{engl2000} \cite{dinh10} \cite{bangti11} \cite{jiang12}.
Most convergence results are established under the well recognised
classical convergence theory for general inverse problems developed in \cite{engl89}.
This classical framework requires the forward map $u(q)$ to be
Fr$\acute{e}$chet differentiable and the Fr$\acute{e}$chet differentive $u'(q)$
Lipschitz continuous. The essence of the classical theory is its source condition
which involves the adjoint operator $u'(q)^*$ and requires the existence of a small source function
in certain sense.
A new convergence theory was proposed in \cite{engl2000} for an inverse
conductivity problem in a parabolic system to relax the restrictive
requirements in the classical convergence theory \cite{engl89}.
A much simpler source condition was presented in \cite{engl2000}, which
involved only the forward map $u(q)$ itself, instead of its derivative and the adjoint, and
does not require the smallness for the source function and the Fr$\acute{e}$chet differentiability
of $u(q)$ and the Lipschitz continuity of the Fr$\acute{e}$chet differentive $u'(q)$.
Same convergence rates as the ones from the classical theory were achieved under these
much weaker and more realistic conditions.
However, this new theory works only to the time-dependent inverse conductivity
problems and does not apply to elliptic inverse problems, and
more importantly, the proposed source conditions can be verified only in the one-dimensional spaces.
Convergence rates of the Tikhonov regularizations were further studied
in \cite{dinh10} for identifying conductivity and radiativity respectively in elliptic systems.
The identifying parameters were assumed to be known over all the boundaries,
then the source conditions in \cite{engl2000} can be relaxed and the convergence rates
can be established for elliptic systems.
But in most applications, the identifying parameters may not be accessible
over the entire boundary.
A novel convergence theory was developed in \cite{bangti11} for general nonlinear
inverse operator equation,
under a special source condition and a strong nonlinearity condition,
which can also get rid of the smallness for the source function.
Inverse conductivity problem was investigated \cite{jiang12} for
a coupled elliptic and parabolic system, and the convergence rate was established
for the $H^1$ regularization and  mixed $L^p$-$H^1$ regularization,
under a simple and easily interpretable source condition, again without smallness
for the source function.
As far as the stationary or instationary inverse conductivity and radiativity problems are concerned,
the aforementioned convergence theories need some source conditions \cite{engl89} or the required nonlinearity
conditions \cite{bangti11}, which are difficult
to be verified in general dimensional spaces, unless adding some restrictive conditions on the
identified parameters or forward solutions.


Variational source condition (VSC) and the resulting convergence rates results were initiated by Hofmann et al. (\cite{HKPS07})
and its extensions were  proven independently in  \cite{BH10}, \cite{Fle10} and \cite{Grasm10}.   In compared to the classical source condition,
VSC does not involve the computation of  Fr\'{e}chet differentiability of the forward operator, and its resulting convergence rates
for the regularized solutions follow immediately from VSC under an appropriate parameter choice rule (see e.g. \cite{HofMat12}).
In this work, we shall first derive some Lipschitz type stability estimates for the proposed inverse problems and then
propose some new variational source conditions to achieve reasonable convergence rates of the Tikhonov regularizations for the inverse problems.
There are three important novelties in this work. The first one
 is its rigorous verification of the proposed VSC in general dimensional spaces under the
Lipschitz type stabilities. The second one is that the identifying radiativities should not be assumed to be known
over the boundaries, which improves the results established in \cite{dinh10}. Finally, it
reveals the relation between the regularity of the radiativities and the convergence rates results.

The remainder of this work is arranged as follows. In Section\,\ref{sec:pre}, some preliminaries
are presented. In sections\,\ref{sec:elliptic},
the conditional stability estimates are derived and some new VSCs are proposed for the elliptic inverse radiativity problem.
We shall verify the VSCs rigorously and dirive the reasonable convergence rates results.
In Section\,\ref{sec:parabolic}, we shall get some conditional stability estimates for the parabolic inverse radiativity problem
and propose some new VSCs to achieve the convergence rates results.
  Some concluding remarks are given in Section\,\ref{sec:conclu}.


\section{Preliminaries}\label{sec:pre}
\setcounter{equation}{0}
In this section, we shall present some preliminaries for our later use.

We first recall some terminologies and notations.  Given  a linear  operator $T:X \to X$ on a complex Hilbert space $X$,
the notations $D(T)$, $\rho(T)$ and $\sigma(T)$ stand for the domain, resolvent and spectrum of $T$ respectively. A linear  operator
$T:D(T)\subset X\to X$ is called closed, if its graph $\{(x,Tx), ~ x \in D(T)\}$ is closed in  $X\times X$.
Furthermore,
the adjoint  of  a densely defined operator $T:D(T)\subset X\to X$  is denoted by $T^*:D(T^*)\subset X\to X$.
We call  $T:D(T)\subset X\to X$ symmetric,  if $Tx=T^*x$ holds true for all $x\in D(T)$, i.e.,
$(Tx,y)_X =( x,Ty)_X $ for all $x,y\in D(T)$. If a symmetric operator $T$ satisfies that $D(T)=D(T^*)$, then
$T$ is said to be self-adjoint.
For two Banach spaces $X$ and $Y$  that are continuously embedded in the same Hausdorff topological vector space, we denote by $[X,Y]_\theta$ ($0\leq\theta\leq 1$)   the complex interpolation space between $X$ and $Y$.

Then for any $s\in(-\infty,\infty)$, we define the following fractional Sobolev space
$$
H^s(\R^d):=\{u\in \mathcal{S}(\R^d)'\mid \|u\|^2_{H^s(\R^d)}:=\int_{\R^d}(1+|\xi|^2)^{s}|(\mathcal{F}u)(\xi)|^2 d\xi<+\infty  \},
$$
where $\mathcal{F}:\mathcal{S}(\R^d)'\to \mathcal{S}(\R^d)'$ is the Fourier transform and $\mathcal{S}(\R^d)'$
denotes the tempted distribution space  (see, e.g.,  \cite{Lions,Wloka,Yagibook}). For a bounded domain $ U \subset \R^d$ with a Lipschitz boundary $\p U$,  the space  $H^s(U)$ with a possibly non-integer exponent $s\geq 0$ is defined as the space of all complex-valued functions $v\in L^2(U)$ satisfying
$V_{\vert U}=v$ for some $V\in H^s(\R^n)$, endowed with the norm
$$
\|v\|_{s,U}:=\inf_{\substack{V_{\vert U}=v \\ V\in H^s(\R^n)}}\|V\|_{H^s(\R^n)}.
$$
{ When no confusion may be caused,
we   simply drop $U$ in the subscription of  $\|\cdot\|_{s,U}$.}
For every $s \in [0,\infty)$, we denote by $\lfloor s \rfloor \in [0,s]$ the largest integer less or equal to $s$. In the case of $s \in (0,\infty)$ with $s= \lfloor s \rfloor + \sigma$ and $0<\sigma<1$, the norm $\|\cdot\|_{s,U}$ is equivalent  to  (cf. \cite{Yagibook})
$$
 \left(\sum_{|\alpha|\leq \lfloor s \rfloor } \|D^\alpha u\|_{ L^2(U) }^2+\sum_{|\alpha|\leq \lfloor s \rfloor } \, \, \iint\limits_{U\times U}\f{|D^\alpha u(x)-D^\alpha u(y)|^2}{|x-y|^{n+2\sigma}} dxdy\right)^{\f{1}{2}}.
$$
 If $s$ is a non-negative integer, then  $H^s(U)$  coincides with the classical Sobolev space.
We set
 $H_0^s(U)$ to the completion of $C_c^\infty(U)$ under the norm $\|\cdot\|_{s,U}$, and $H^{-s}(U)$  to the dual space of
 $H_0^s(U)$  with respect to inner product of $L^2(U)$. It is also well-known that
 the inner product $( \cdot, \cdot)_{U}=\int_{U} f\overline{g} dx $  extends to an bounded
 sesquilinear form on  $H^{-s}(U)\times H^s_0(U)$,  where $\overline{g}$ denotes the complex conjugate of $g$,
 which satisfies $|\langle f,g\rangle_{H^{-s}(U),H^s_0(U)} |\leq
 \|f\|_{H^{-s}(U) }\|g\|_{H^s_0(U)} $ for all $f\in H^{-s}(U)$ and $g\in H^s_0(U)$.

Throughout the paper,  $C$ is often used for a generic positive constant.  We shall often use the symbol $\langle\cdot,\cdot\rangle$
for general duality pairing, and denote by $\rightarrow$
and $\rightharpoonup$ the strong convergence and weak convergence respectively.

%

\begin{lemma}[\cite{Tambaca,Yagibook}]\label{lemma:product}
\begin{enumerate}
\item Let $r,s\in \R$ such that $r,s<\f{d}{2}$ and $r+s>0$. Then
for $t=r+s-\f{d}{2}$ and distributions $u\in H^r(\R^d)$,
$v\in H^s(\R^d)$ one has $uv\in H^t(\R^d)$ and the following inequality holds
$$
\|uv\|_{t,\R^d}\leq C\|u\|_{s,\R^d}\|v\|_{r,\R^d}.
$$

\item
Let  $r>\f{d}{2}$. Then
$H^s(\Omega)$ is an algebra under pointwise multiplication, i.e.,  for all functions $u,v\in H^r(\Omega)$, it holds
$$
\|uv\|_{H^r(\Omega)}\leq C \|u\|_{H^r(\Omega)}\|v\|_{H^r(\Omega)}.
$$

\end{enumerate}

\end{lemma}

We end this section by recalling the following two well-posedness results, which can be found,
e.g., \cite{gri85}  (Corollary 2.2.2.4) and   \cite{Ladyzenskaja} (Chapter VI, section 9)  for the elliptic system (\ref{q1})
and  parabolic system \eqref{q1p} respectively.
\begin{lemma}\label{lem:well}
Assume that $a(\x)\in W^{1,\infty}(\Om)$ with a positive lower bound,
$q(\x)\in K$,  $f(\x)\in L^2(\Om)$ and $g(\x)\in H^{\frac{3}{2}}(\p\Om)$.
Then there exists a unique solution $u\in H^{2}(\Om)$ to the system (\ref{q1})
with the estimate
\beqn
&&\|u\|_{2,\Om}
\leq C(\|f\|_{0,\Om}+\|g\|_{\frac{3}{2},\p\Om}).
\label{eadd2}
\eqn
\end{lemma}

%
\begin{lemma}\label{lem:wellt}
Assume that  $a(\x)\in W^{1,\infty}(\Om)$  with a positive lower bound,
 $q(\x)\in K$, $f(\x,t)\in L^2(0,T;L^2(\Om))$, $g(\x,t)\in L^2(0,T;H^{\frac{3}{2}}(\p\Om))\cap H^{\f 3 4}(0,T;L^2(\p\Omega))$
 and $u_0(\x)\in H^1(\Om)$.
Then there exists a unique solution $u\in L^2(0,T;H^{2}(\Om))\cap H^1(0,T;L^{2}(\Om))$ to
the system (\ref{q1p}) with the estimate:

\beqn
&&\|u\|_{L^2(0,T;H^{2}(\Om))}+\|u\|_{H^1(0,T;L^{2}(\Om))}\nb\\
&\leq& C(\|f\|_{L^2(0,T;L^{2}(\Om))}+\|g\|_{L^2(0,T;H^{\frac{3}{2}}(\p\Om))}
+\|g\|_{ H^{\f 3 4}(0,T;L^2(\p\Omega))}
+\|u_0\|_{1,\Om}).
\label{padd2}
\eqn
\end{lemma}

\section{Convergence rates of Tikhonov regularization for
elliptic inverse radiativity problem }\label{sec:elliptic}
\setcounter{equation}{0}

We will study in this section the Lipschitz type  stability and convergence rates of the Tikhonov
regularization for the recovery of the radiativity in the elliptic system  (\ref{q1}).  Throughout this section, we always assume that
\begin{enumerate}
\item[]   $a(\x)\in W^{1,\infty}(\Om)$  has a positive lower bound, $f(\x)\in L^2(\Om)$ and $g(\x)\in H^{\frac{3}{2}}(\p\Om)$
in \eqref{q1}.
\end{enumerate}

\subsection{Measurement data in gradient form}

Now suppose that the measurement data $\nabla z^\delta$ of $\nabla u(q)$ is noisy in $\Om$, with a noise level $\delta$, namely
\bb \label{eq:data}
\|\nabla u(\qd)- \nabla z^\delta\|_{0,\Om} \leq\delta\,,
\ee
where  $\qd$ is the true physical radiativity. The elliptic inverse radiativity problem  is highly ill-posed \cite{engl96},
and is usually transformed into an effective and stable minimisation system  with Tikhonov regularization:
\bb
\min_{q\in K}J_{\delta,\beta}(q)=\min_{q\in K}\Big(\f{1}{2}\|\nabla u(q)-\nabla z^\delta\|_{0,\Om}^2
+ \f \beta  2 \|q- q^*\|^2_{0,\Om}\Big),
\label{dai3}
\ee
where $\beta>0$ is the regularization parameter and $q^*\in K$ is an a priori estimate of the true
parameter $\qd$. 

We refer to \cite{dinh10} (Theorem 3.1) and establish the following theorem for the existence
of minimizers to optimization problem \eqref{dai3}.
\begin{theorem}\label{lemma:exis}
 There exists at least a minimizer  $\qda$ to optimization problem \eqref{dai3}.
\end{theorem}

In the following, we shall first prove a Lipschitz type stability estimate for the elliptic inverse radiativity problem,
which is of fundamental importance in the verification of the VSC.

\begin{theorem}\label{lemma:stability}

Assume $|u(\qd)|\geq c_0$ for some positive constant $c_0$ in $\Om$, then for any $\epsilon\in (0,\frac{1}{2})$, we have
\beqn\label{stability}
\|q-\qd\|_{H^{-1-\epsilon}(\Omega)}\leq C \|u(q)-u(\qd)\|_{1,\Omega} \quad \forall q\in K.
\eqn

%

\end{theorem}

\begin{proof}
We know easily from system \eqref{q1} that for all $q\in K$,
\begin{equation}\label{eq00}
-\nabla\cdot( a(\x)\nabla( u(\qd)-u(q)))+q(u(\qd)-u(q))
=u(\qd)(q-q^\dag).
\end{equation}
Since  $(u(\qd)-u(q))\in H^1_0(\Om)$,  we can multiply \eqref{eq00} by a function $\varphi\in H_0^1(\Omega)$ and obtain that
\begin{align*}
\left|\int_\Omega u(\qd)(q-q^\dag)\varphi dx \right|
\leq  \overline{q}\int_\Omega |(u(\qd)-u(q))\varphi| dx
+\|a\|_{L^\infty(\Omega)}\|\nabla (u(\qd)-u(q))\cdot \nabla \varphi \|_{0,\Omega},
\end{align*}
which implies that
\beqn\label{eq:almost}
c_1\|u(\qd)(q-q^\dag)\|_{H^{-1}(\Omega)}\leq \|u(q)-u(\qd)\|_{H_0^1(\Omega)}.
\eqn
for some constant $c_1>0$.
By the definition of the $H^{-1-\epsilon}(\Om)$ norm, we have
\beqn\label{eq:inner}
 \|q-q^\dag\|_{H^{-1-\epsilon}(\Omega)}&=&\sup_{\|\varphi\|_{H_0^{1+\epsilon}(\Omega)}=1}
\left|\int_\Omega (q-q^\dag)\varphi(x) d\x \right|\nb\\&=&
\sup_{\|\varphi\|_{H_0^{1+\epsilon}(\Omega)}=1}
\left|\int_\Omega u(\qd)(q-q^\dag)\frac{\varphi(x)}{u(\qd)} d\x \right|.
\eqn
Since $u(\qd)\in H^2(\Omega)$ (by Lemma \ref{lem:well}) with  $|u(\qd)|\geq c_0$,
it follows that $\f{1}{u(\qd)}\in H^2(\Omega)$ by Leibniz' rule.

Now let $w\in H^2(\R^d)$ be an extension of $\f{1}{u(\qd)}$ such that $\|w\|_{H^2(\R^d)}\leq 2\|\f{1}{u(\qd)}\|_{2,\Omega}$, then
$\varphi /u(\qd)=w\varphi \mid_{\Omega}$.  For the space dimension $d=3$,    one has
$1+\epsilon+(\f{3}{2}-\f{\epsilon}{2})-\f{d}{2}>1$ and $1+\epsilon\leq\f{d}{2}$ for
any $\epsilon\in (0,1/2)$, and we have by  using Lemma
\ref{lemma:product} (1) that
\beqn\label{eq:product}
\|w\varphi\|_{1,\R^d}\leq\|w\varphi\|_{1+\f \epsilon 2+\f{3}{2}-\f{d}{2},\R^d}\leq C \|w\|_{\f{3}{2},\R^d}\|\varphi\|_{H^{1+\epsilon}_0(\R^d)}
\leq C\|\f{1}{u(\qd)}\|_{2,\Omega}\|\varphi\|_{H^{1+\epsilon}_0(\Omega)}.
\eqn
where we have used the continuous embedding $H^{\f 3 2}(\R^d)\Subset H^{\f{3-\epsilon}{2}}(\R^d)$.
For the space dimension $d=2$,  ${3}/{2}>{d}/{2}$ and $1+\epsilon>{d}/{2}$ for
any $\epsilon\in (0,1/2)$. Then we obtain by Lemma \ref{lemma:product} (2) that for any $r\in(1,\min\{3/2,1+\epsilon\})$,

\begin{align}\label{eq:product2}
\|w\varphi\|_{1,\Om}\leq\|w\varphi\|_{r,\Om}\leq C \|w\|_{r,\Om}\|\varphi\|_{r,\Om}
\leq& C\|w\|_{\f{3}{2},\Om}\|\varphi\|_{H^{1+\epsilon}_0(\Omega)}\notag \\
\leq& C\|\f{1}{u(\qd)}\|_{2,\Omega}\|\varphi\|_{H^{1+\epsilon}_0(\Omega)}.
\end{align}
Hence, from \eqref{eq:inner}, \eqref{eq:product} and \eqref{eq:product2}
and noting that
$\f{1}{u(\qd)}\in H^2(\Om)$ and $\|\varphi\|_{H^{1+\epsilon}_0(\Omega)}=1$, we get
\beqn
&&\|q-q^\dag\|_{H^{-1-\epsilon}(\Omega)}\leq \|u(\qd)(q-q^\dag)\|_{H^{-1}(\Om)}\|\frac{\varphi(x)}{u(\qd)}\|_{1,\Om}\nb\\
&\leq&C\|u(\qd)(q-q^\dag)\|_{H^{-1}(\Om)}\|\f{1}{u(\qd)}\|_{2,\Omega}\|\varphi\|_{H^{1+\epsilon}_0(\Omega)}\nb\\
&\leq&C\|u(\qd)(q-q^\dag)\|_{H^{-1}(\Om)}.
\label{aj1}
\eqn
This together with \eqref{eq:almost} implies that
\beqnx
\|q-q^\dag\|_{H^{-1-\epsilon}(\Omega)}\leq C\|u(\qd)(q-q^\dag)\|_{H^{-1}(\Om)}
\leq C\|u(q)-u(\qd)\|_{H_0^1(\Omega)}.
\eqnx

\end{proof}

\begin{remark}\label{remark:product}
By the arguments leading to \eqref{eq:product} and \eqref{eq:product2}, we can actually prove that if $w\in H^{3/2}(\Omega)$
and $\epsilon\in (0,1/2)$, then
$$
\|w\varphi \|_{1,\Omega}\leq C\|w\|_{3/2,\Omega}\|\varphi\|_{1+\epsilon,\Omega}\quad \forall \varphi\in C_0^\infty(\Omega).
$$

\end{remark}

\begin{remark}
Since the embedding $H^{-s}(\Omega)\Subset H^{-t}(\Omega)$ is continuous whenever  $t>s\geq 0$, we can obtain that for all $s>1$
and $q\in K$, it holds
\beqnx
\|q-\qd\|_{H^{-s}(\Omega)}\leq C \|u(q)-u(\qd)\|_{1,\Omega} .
\eqnx

\end{remark}

\begin{remark}
As $u(q)-u(\qd)\in H^1_0(\Om)$, then by the Poinc$\acute{a}$re's inequality, we have for all $s>1$ and $q\in K$,
\bb
\|q-q^\dag\|_{H^{-s}(\Omega)}\leq C\|u(q)-u(\qd)\|_{H_0^1(\Omega)}\leq
C\|\nabla u(q)-\nabla u(\qd)\|_{0,\Omega}.
\ee
\end{remark}

Now, we are going to propose a variational source condition, which shall be verified rigorously.
The crucial {\bf variational source condition} is proposed as follows:
\beqn\label{VSC}
\f{1}{4}\|q-q^\dag\|^2_{0,\Omega}
\le \f{1}{2}\|q-q^*\|_{0,\Omega}^2-\f{1}{2}
\|q^\dag-q^*\|_{0,\Omega}^2+C\|u(q)-u(q^\dag)\|_{1,\Omega}^\alpha \quad \forall\,q\in K,
\eqn
where $\alpha$ is selected in Theorem \ref{the:vsc}. Then
using parallelogram law in Hilbert spaces, it is easy to see that \eqref{VSC} is equivalent to
the following inner product form:

\beqn\label{inner}
( q^\dag-q^*,q^\dag-q)_{\Omega}
\le \f{1}{4}\|q-q^\dag\|^2_{0,\Omega}
+C\|u(q)-u(q^\dag)\|_{1,\Omega}^\alpha  \quad \forall\,q\in K.
\eqn

Before verifying condition \eqref{inner}, we still present some preliminaries.
Assume that $\A:=-\Delta$  with domain
$D(\A)=H_0^1(\Omega)\cap H^2(\Omega)$.  It is well-known that the operator $\A: D(\A)\subset L^2(\Omega)\to L^2(\Omega)$  is   densely defined, closed, self-adjoint and $m$-accretive.  Then, in view of the compactness of the embedding $D(\A)\subset  L^2(\Om)$, we infer that there exists a
complete orthonormal basis $\{ e_n\}_{n=1}^\infty \subset L^2(\Om)$ such that
\beqn
(\A u,u)_{L^2(\Omega)}=\sum_{n=1}^\infty \lambda_n |(u,e_n)_{L^2(\Omega)}|^2  \quad \forall\, u\in D(\A),
\eqn
where $\lambda_n$ are the eigenvalues of $\A$ satisfying
$0<\lambda_{1}\leq \lambda_2 \leq \cdots$, $\lim_{n\to\infty} \lambda_n=+\infty$, and for any $n$, $e_n $ is the eigenfunction
of $\A$ for the eigenvalue of $\lambda_n$, i.e.,
$\A e_n=\lambda_n  e_n$.
For every $\theta \in \R$, the fractional power $\A^\theta$ of $\A$ can be defined as
\beqn\label{def:As}
\A^\theta  u:=\sum_{n=1}^\infty \lambda^\theta_n  (u, e_n)_{L^2(\Om)} e_n \q \forall\, u\in D(\A),
\eqn
where the domain $D(\A^\theta )$ is given by
\beqn\label{eq:Bs}
D(\A^\theta)=\{ u\in L^2(\Omega)\mid \,\, \sum_{n=1}^\infty \lambda_n^{2\theta} |(u,e_n)_{L^2(\Om)}|^2 <\infty\}.
\eqn
Moreover,  $\A^\theta: D(\A^\theta)\subset L^2(\Om) \to   L^2(\Om)$ is also self-adjoint, and $D(\A^\theta)$ is a Banach space equipped with the norm
\beqn\label{eq:Asnorm}
\| u \|_{D(\A^\theta )}:=\|\A^\theta  u\|_{0,\Omega}=
\left(\sum_{n=1}^\infty \lambda_n^{2\theta}|(u,e_n)_{L^2(\Omega)}|^2\right)^{1/2} \quad \forall u\in D(\A^\theta),
\eqn
which is also equivalent to the corresponding graph norm of $(\A^\theta,D(\A^\theta))$ (for more details, we refer to \cite{Wloka}).
Let us mention that for all $\theta\in [0,1/4)\bigcup(1/4,1/2]$, it holds that (see \cite{Lions})
\beqn\label{eq:normeq0}
D(\A^{\theta})=H^{2\theta }_0(\Omega).
\eqn

We are now ready to verify the equavilent {\bf variational source condition}  \eqref{inner}.

\begin{theorem}\label{the:vsc}
Assume $|u(\qd)|\geq c_0$ in $\Om$ and  $q^\dag-q^*\in H_0^\kappa(\Omega)$ with $\kappa>0$ and $\kappa\neq 1/2$,
then VSC \eqref{inner} holds with  some parameter $\alpha$ such that
$$
\begin{cases}
\alpha=1 \quad &\text{if}\,\, \kappa>1,\\
\alpha<\f{2\kappa}{1+\kappa}\,\textup{but it can be choosen arbitrarily close to}\,\,\f{2\kappa}{1+\kappa}
\quad &\text{if}\,\, \kappa\in(0,\frac{1}{2})\cup (\frac{1}{2},1].\\
\end{cases}
$$

\end{theorem}

\begin{proof}
	
Firstly, it is immediately to see that \eqref{inner} holds if $\qd-q^*=0$.  In the sequel, we
shall consider the case when $\qd-q^*\neq 0$.

Now if $\qd-q^*\neq 0 $ and $\kappa>1$, then by making use of  Theorem \ref{lemma:stability},
we have
\beqnx
&& |(\qd-q^*,\qd-q)_{\Omega}|\leq \|\qd-q^*\|_{H_0^\kappa(\Omega)}
 \|\qd-q\|_{H^{-\kappa}(\Omega)}\\&\leq& C  \|\qd-q^*\|_{H_0^\kappa(\Omega)}\|u(\qd)-u(q)\|_{1,\Omega}
 \leq C  \|u(\qd)-u(q)\|_{1,\Omega},
\eqnx
which verifies \eqref{inner} with $\alpha=1$.

Next, we start to consider the case when $\kappa\in (0,\frac{1}{2})\cup (\frac{1}{2},1]$.
For each $\lambda>0$,  we define a family of orthogonal projections
$$
P_\lambda u:=\sum_{\lambda_n<\lambda}(u,e_n)_{\Omega}e_n.
$$
And if $\lambda<\lambda_1$, we set $P_\lambda=0$.
Then the Young's inequality yields
\begin{align}\label{eq:vsc:00}
|((I-P_\lambda)(\qd-q^*),\qd-q)_{\Omega}|
\le\f{\|\qd-q\|^2}{4}+ \|(I-P_\lambda)(\qd-q^*)\|^2_{0,\Omega}.
\end{align}
As $q^\dag-q^*\in H_0^\kappa(\Omega)$, we have from
\eqref{eq:normeq0} that $\qd-q^*\in D(\A^{\kappa/2})$.
Hence, by the definition of $P_\lambda$, it is readily to see that
\beqn\label{eq:vsc:0}
&&\|(I-P_\lambda)(\qd-q^*)\|_{0,\Omega}^2
=\sum_{\lambda_n\geq \lambda  }|(\qd-q^*,e_n)_{\Omega}|^2\notag\\
&\leq&\f{\sum_{n\geq 1} \lambda_n^{\kappa }|(\qd-q^*,e_n)|^2}{\lambda^{\kappa }}=
\f{\|\qd-q^*)\|_{D(A^{\kappa/2})}^2}{\lambda^{\kappa }}.
\eqn
On the other hand,  for any $s>1$, it yields by Theorem  \ref{lemma:stability} that
\beqn\label{eq:vsc:1}
 |(P_\lambda(\qd-q^*),\qd-q)_{\Omega}|&\leq& \|(P_\lambda(\qd-q^*)\|_{H^{s}_0(\Omega)}\|\qd-q\|_{H^{-s}(\Omega)}\nb\\
 &\leq& C\|(P_\lambda(\qd-q^*)\|_{H^{s}_0(\Omega)}\|u(\qd)-u(q)\|_{1,\Omega}.
\eqn
We then estimate $\|(P_\lambda(\qd-q^*)\|_{H^{s}_0(\Omega)}$. Indeed, by \eqref{eq:normeq0} one has
\begin{align}
\|(P_\lambda(\qd-q^*)\|^2_{H^{s}_0(\Omega)}\leq& C
\|(P_\lambda(\qd-q^*)\|^2_{D(\A^{\f{s}{2}})}
=\sum_{\lambda_n<\lambda}\lambda_n^{s}  |(\qd-q^*,e_n)|^2\notag\\
=&\sum_{\lambda_n<\lambda}\lambda_n^{{s-\kappa}}\cdot \lambda_n^{\kappa } |(\qd-q^*,e_n)|^2\notag \leq \lambda^{{s-\kappa}}\|\qd-q^*\|_{D(\A^{\kappa/2})}^2,\notag
\end{align}
which, togother with \eqref{eq:vsc:1}, implies
\beqn\label{eq:vsc:02}
 |(P_\lambda(\qd-q^*),\qd-q)_{\Omega}|\leq C  \lambda^{\f{s-\kappa}{2}}\|\qd-q^*\|_{D(\A^{\kappa/2})}\|u(\qd)-u(q)\|_{1,\Omega}.
\eqn
Combing \eqref{eq:vsc:00},\eqref{eq:vsc:0} and \eqref{eq:vsc:02}, we have
\begin{align}\label{vsc:final}
(\qd-q^*,\qd-q)_{\Omega}\leq& \f{\|\qd-q\|^2}{4}\\
&+CA\inf_{\lambda>0}
\left(\f{A}{\lambda^{{\kappa} }}+\lambda^{\f{s-\kappa}{2}}\|u(\qd)-u(q)\|_{1,\Omega}\right)\notag
\quad \forall\, q\in K,
\end{align}
 where $A=\|(\qd-q^*)\|_{D(\A^{\kappa/2})}$.  Choosing $\f{A}{\lambda^{ {\kappa}}}=\lambda^{\f{s-\kappa}{2}}\|u(\qd)-u(q)\|_{1,\Omega}$,
i.e.,  $\lambda^{\f{s+\kappa}{2}}=\f{A}{\|u(q^\dag)-u(q)\|_{1,\Omega}}$ in
\eqref{vsc:final},  we know that
\beqn
(\qd-q^*,\qd-q)_{\Omega}\leq\f{\|\qd-q\|^2}{4}+2C A A^{\f{s-\kappa}{s+\kappa}}
\|u(\qd)-u(q)\|^{\f{2\kappa}{s+\kappa}} \quad \forall\, q\in K.
\eqn
Since $s>1\geq k>0$, then we have $ A A^{\f{s-\kappa}{s+\kappa}}\leq C$ and
$\f{2\kappa}{s+\kappa}<\f{2\kappa}{1+\kappa}$, which completes the proof.

\end{proof}

\begin{remark}

Recalling the convergence rates Theorem 3.4 in \cite{dinh10}, they assumed $\frac{q^\dag-q^*}{u(q^\dag)}\in H^1(\Omega)$
and obtained the convergence rate $\|\qda-q^\dag\|_{0,\Omega}=O(\sqrt{\delta})$. But to verify the proposed
source condition (3.21) in \cite{dinh10}, they should assume the rediativity $q^\dag$ to be known on the whole boundary.
However, it is well-known that
$H_0^\kappa(\Omega)=H^\kappa(\Omega)$  when $\kappa\in (0,1/2)$ (See e.g. \cite[Theorem1.40]{Yagibook}). Therefore,
when $\kappa\in (0,1/2)$
we don't need any priori knowledge of $\qd$ on the boundary, but only assume that $q^\dag-q^*\in H^\kappa(\Omega)$.

\end{remark}

We end this section by establishing the following theorem for the summarization of
the main convergence rates results.


\begin{theorem}\label{the:ellip:con}
Assume $|u(\qd)|\geq c_0$ in $\Om$ and   $q^\dag-q^*\in H_0^\kappa(\Omega)$ with $\kappa>0$ and $\kappa\neq 1/2$,
 and $\alpha$ is the
parameter chosen as in Theorem \ref{the:vsc},  then we have the following convergence rates
\beqn\label{eq:convergence0}
\|\nabla u(\qda)-\nabla u(\qd)\|_{0,\Omega}=O(\delta) 
\eqn
and
\beqn\label{eq:convergence1}
\|\qda-\qd\|_{0,\Omega}=O(\delta^{\f{\alpha}{2}}) 
\eqn
under the parameter choice  $\beta=\delta^{2-\alpha}$.
\end{theorem}

\begin{proof}
%
By the definition of $\qda$ in \eqref{dai3} and using \eqref{eq:data}, we have
\begin{align}\label{eqc}
\f{1}{2}\|\nabla u(\qda)-\nabla z^\delta\|_{0,\Om}^2+ \f \beta 2
\|\qda -q^*\|_{0,\Omega}^2\le& \f{1}{2}\|\nabla u(\qd)-\nabla z^\delta\|_{0,\Om}^2+\f  \beta 2
\|q^\dag-q^*\|_{0,\Omega}^2\notag\\
\le& \f 1 2 \delta^2+ \f \beta 2
\|q^\dag-q^*\|_{0,\Omega}^2,
\end{align}
which implies
\begin{align}\label{eqb}
\f{1}{2}
\|\qda-q^*\|^2_{0,\Omega}-
\f{1}{2}\|q^\dag-q^*\|_{0,\Omega}^2\le& \f{\delta^2}{2\beta}-\f{1}{2\beta}  \|\nabla u(\qda)-\nabla z^\delta\|_{0,\Om}^2
\leq  \f{\delta^2}{2\beta}.
\end{align}
Using  \eqref{VSC}, \eqref{eqb} and triangle inequality, we have
\begin{align}\label{eqd}
0\le& \f{1}{2}\|\qda-q^*\|_{0,\Omega}^2-\f{1}{2}
\|\qd-q^*\|_{0,\Omega}^2+C\|u(\qda)-u(q^\dag)\|_{1,\Om}^\alpha \notag \\
\le& \f{1}{2\beta}\left(\delta^2-\|\nabla u(\qda)-\nabla z^\delta\|_{0,\Om}^2\right)+C\|u(\qda)-u(q^\dag)\|_{1,\Om}^\alpha\nb\\
\le& \f{1}{2\beta}\left(2\delta^2-\frac{1}{2}\|\nabla u(\qda)-\nabla u(\qd)\|_{0,\Om}^2\right)+C\|u(\qda)-u(q^\dag)\|_{1,\Om}^\alpha.
\end{align}
As $u(\qda)-u(\qd)\in H^1_0(\Om)$, then by the Poinc$\acute{a}$re's inequality, we have
\beqnx
\|u(\qda)-u(\qd)\|_{1,\Omega}\leq
C\|\nabla u(\qda)-\nabla u(\qd)\|_{0,\Omega},
\eqnx
which together with \eqref{eqd} implies that
\beqn\label{eq:key}
\|\nabla u(\qda)-\nabla u(\qd)\|_{0,\Omega}^2
\le 4\delta^2  +C\beta \|\nabla u(\qda)-\nabla u(\qd)\|_{0,\Omega}^\alpha.
\eqn

Now if $ \|\nabla u(\qda)-\nabla u(\qd)\|_{0,\Omega}<\delta$, then one has proved the convergence
rate \eqref{eq:convergence0}. Otherwise, if
 $ \|\nabla u(\qda)-\nabla u(\qd)\|_{0,\Omega}\geq \delta$, as $\alpha\leq 1$, then one has
 $$
  \|\nabla u(\qda)-\nabla u(\qd)\|_{0,\Omega}^\alpha
  \leq   \|\nabla u(\qda)-\nabla u(\qd)\|_{0,\Omega} \delta^{\alpha-1}.
 $$
Taking the above inequality into \eqref{eq:key} and choosing $\beta=\delta^{2-\alpha}$, we get
\beqnx
\|\nabla u(\qda)-\nabla u(\qd)\|_{0,\Omega}^2
&\le& 4\delta^2  +C\delta^{2-\alpha}\|\nabla u(\qda)-\nabla u(\qd)\|_{0,\Omega} \delta^{\alpha-1}\\
&=&4\delta^2  +C\delta\|\nabla u(\qda)-\nabla u(\qd)\|_{0,\Omega}\\
&\leq&(4+C)\delta\|\nabla u(\qda)-\nabla u(\qd)\|_{0,\Omega},
\eqnx
which implies $\|\nabla u(\qda)-\nabla u(\qd)\|_{0,\Omega}=O(\delta)$. Therefore,
  \eqref{eq:convergence0} holds.

Finally, using  \eqref{VSC}, \eqref{eqb} and Poinc$\acute{a}$re's inequality, we obtain
\beqnx
\f{1}{4}\|\qda-q^\dag\|^2_{0,\Omega}
&\le& \f{1}{2}\|\qda-q^*\|_{0,\Omega}^2-\f{1}{2}
\|q^\dag-q^*\|_{0,\Omega}^2+C\|u(\qda)-u(q^\dag)\|_{1,\Omega}^\alpha\\
&\le& \frac{\delta^2}{2\beta}+C\|\nabla u(\qda)-\nabla u(\qd)\|_{0,\Omega}^\alpha.
\eqnx
Then choosing $\beta=\delta^{2-\alpha}$ and using \eqref{eq:convergence0}, we have
\beqnx
\f{1}{4}\|\qda-q^\dag\|^2_{0,\Omega}
\le \frac{\delta^2}{2\delta^{2-\alpha}}+C\delta^\alpha\leq C\delta^\alpha,
\eqnx
which verifies \eqref{eq:convergence1}.

\end{proof}

\begin{corollary}\label{coro:ellip}
Under the hypothesises and settings of Theorem \ref{the:ellip:con} and assume that $2\leq p<+\infty$, we then have the convergence rate
\beqn\label{eq:convergence12}
\|\qda-\qd\|_{L^p(\Omega)}=O(\delta^{\f{\alpha}{p}}) 
\eqn
under the parameter choice  $\beta=\delta^{2-\alpha}$.

\end{corollary}
\begin{proof}
Since $\|\qda-\qd\|_{L^\infty(\Omega)}\leq  2\bar q$ for all $\qda\in K$, we can obtain by H\"{o}lder's inequality that for any $2\leq p<+\infty$,
$$
\|\qda-\qd\|_{L^p(\Omega)}\leq (2\bar q)^{\f{p-2}{p}} \|\qda-\qd\|_{L^2(\Omega)}^{\f{2}{p}},
$$
which, together with Theorem \ref{the:ellip:con}, completes the proof.
\end{proof}

\subsection{Measurement data in $L^2$-norm}
In this subsection, we aim at recovering $q(\x)$ from the $L^2$-noisy data of $u(\qd)$.  That is, we assume that the measurable data $z^\delta$ of $u(q)$ is noisy in $L^2(\Om)$ with a noise level $\delta$, namely
\bb\label{eq:datal2}
\| u(\qd)-  z^\delta\|_{0,\Om} \leq\delta\,,
\ee
where  $\qd$ is the true physical radiativity. The elliptic inverse radiativity problem  is transformed into an effective and stable minimisation system  with Tikhonov regularization:
\bb
\min_{q\in K}J_{\delta,\beta}(q)=\min_{q\in K}\Big(\f{1}{2}\| u(q)-z^\delta\|_{0,\Om}^2
+ \f \beta  2 \|q- q^*\|^2_{0,\Om}\Big),
\label{dai32}
\ee
where $\beta>0$ is the regularization parameter and $q^*\in K$ is an a priori estimate of the true solution.
In this subsection, we denote by $\qda $ the
minimizer of \eqref{dai32}.  One will see that if we lose the information of $\nabla z^\delta$, then the convergence rate of  $\qda$ is
slower (See Theorem \ref{the:ellip:con:l2} below).
To study the convergence rate of  $\qda$, we need the H\"{o}lder type stability estimate for the elliptic inverse radiativity problem.

\begin{theorem}\label{lemma:stability:l2}

Assume $|u(\qd)|\geq c_0$ in $\Om$, then for any $\epsilon\in (0,\frac{1}{2})$, we have
\beqn\label{stabilityl2}
\|q-\qd\|_{H^{-1-\epsilon}(\Omega)}\leq C \|u(q)-u(\qd)\|_{0,\Omega} ^{\f{1}{2}}\quad \forall q\in K.
\eqn

%

\end{theorem}
\begin{proof}
From Theorem \ref{lemma:stability}  and the well-known Gagliardo-Nirenberg interpolation inequality
\beqn\label{Sob:interp}
\|u\|_{1,\Omega}\leq C\|u\|_{2,\Omega}^{\f  1 2} \|u\|_{0,\Omega}^{\f 1 2}\quad \forall\, u\in H^2(\Omega),
\eqn
it suffices to show that
\beqn\label{bound:H2}
 \|u(q)-u(\qd)\|_{2,\Omega}\leq C\quad \forall q\in K.
\eqn

In view of \eqref{eq00}, we know that
\beqnx
-\nabla\cdot( a(\x)\nabla( u(\qd)-u(q)))+q(u(\qd)-u(q))=u(\qd)(q-q^\dag).
\eqnx
Then by Lemma \ref{lem:well}, we have
\beqn\label{eq223}
\| u(\qd)-u(q)\|_{2,\Omega}\leq C\|u(\qd)(q-q^\dag)\|_{0,\Omega}
\leq 2C\overline{q}\|u(q^\dag)\|_{0,\Omega},
\eqn
which yields \eqref{bound:H2}.
\end{proof}
Then we can prove the following analogue of Theorem \ref{the:vsc}, whose proof is the same except
that we use the results in Theorem  \ref{lemma:stability:l2} instead of Theorem \ref{lemma:stability}.

\begin{theorem}\label{the:vsc:l2}
Assume $|u(\qd)|\geq c_0$ in $\Om$ and   $q^\dag-q^*\in H_0^\kappa(\Omega)$ with $\kappa>0$ and $\kappa\neq 1/2$,
then VSC \beqn\label{VSC:l2}
\f{1}{4}\|q-q^\dag\|^2_{0,\Omega}
\le \f{1}{2}\|q-q^*\|_{0,\Omega}^2-\f{1}{2}
\|q^\dag-q^*\|_{0,\Omega}^2+C\|u(q)-u(q^\dag)\|_{0,\Omega}^{\alpha } \quad \forall\,q\in K,
\eqn
holds with  some constant $C>0$ and parameter  $\alpha$ such that
$$
\begin{cases}
\alpha=1/2 \quad &\text{if}\,\, \kappa>1,\\
\alpha<\f{\kappa}{1+\kappa}\,\textup{but it can be choosen arbitrarily close to}\,\,\f{\kappa}{1+\kappa}
\quad &\text{if}\,\, \kappa\in(0,\frac{1}{2})\cup(\frac{1}{2},1].\\
\end{cases}
$$

\end{theorem}

With the aid of Theorem \ref{the:vsc:l2}, we can establish the convergence of $u(\qda)$ as follows.  Its proof is the same as in
Theorem \ref{the:vsc:l2}, and we only need to replace $\|u(\qda)-u(q^\dag)\|_{1,\Omega}$ (or $\|\nabla u(\qda)-\nabla u(q^\dag)\|_{0,\Omega}$)
with $\|u(\qda)- u(\qd)\|_{0,\Omega}$, which is valid due to \eqref{VSC:l2}.

\begin{theorem}\label{the:ellip:con:l2}
Assume $|u(\qd)|\geq c_0$ in $\Om$ and   $q^\dag-q^*\in H_0^\kappa(\Omega)$ with $\kappa>0$ and $\kappa\neq 1/2$, and $\alpha$ is the
parameter chosen as in Theorem \ref{the:vsc:l2},  then we have the following convergence rates
\beqnx
\|u(\qda)- u(\qd)\|_{0,\Omega}=O(\delta) 
\eqnx
and
\beqnx
\|\qda-\qd\|_{0,\Omega}=O(\delta^{\f{\alpha}{2}}) 
\eqnx
under the parameter choice  $\beta=\delta^{2-\alpha}$.
\end{theorem}

\begin{remark}
Following the arguments used in Theorem \ref{the:ellip:con} and Theorem \ref{the:ellip:con:l2},
we can actually prove a general result. Let us consider the following Tikhonov regularization
\bb
\min_{q\in K}J_{\delta,\beta}(q)=\min_{q\in K}\Big(\f{1}{2}\| u(q)-z^\delta\|_{Y}^2
+ \f \beta  2 \|q- q^*\|^2_{0,\Om}\Big),
\label{dai321}
\ee
where the noisy data $z^\delta\in Y$ satisfies
\beqnx
\|z^\delta-u(\qd) \|_Y \leq \delta.
\eqnx
If the following conditional stability estimate holds: there exists some $s_0\in (0,3/2)$ and $\alpha_0 \geq 0$ such that
\beqn\label{condi:gel}
\|q^\dag-q\|_{H^{-s_0}(\Omega)} \leq C \|u(q^\dag)-u(q)\|^{\alpha_0}_Y\quad \forall\, q\in  K.
\eqn
Then the assumption 
 $q^\dag-q^*\in H^{s_0\theta}(\Omega)$ for some $\theta\in (0,1]$ with
$\theta s_0\neq \f{1}{2}$ can imply that  the minimizer $\qda$ of \eqref{dai321}
enjoys convergence rates $\|\qda-\qd\|_{0,\Omega}=O(\delta^{\alpha/2})$  with $\alpha=\alpha_0\theta$ under a priori parameter choice
$\beta=\delta^{2-\alpha}$.  In conclusion, the space $Y$ in \eqref{condi:gel}  quantifies the regularity assumptions 
on the noisy measurable data $y^\delta$ and $\alpha_0$ is the ``maximal'' convergence rate when $q^\dag-q^*\in H^{s_0}_0(\Omega)$. 
If the regularity of $q^\dag-q^*$ is weaker, then the convergence rate is slower.
\end{remark}

\section{Convergence rates of Tikhonov regularization for
parabolic inverse radiativity problem }\label{sec:parabolic}

In this section, we shall study the inverse problem of recovering
 radiativity in the parabolic system \eqref{q1p} from a partial measure over
$\Omega\times I$, where $I$ is an open subinterval of $(0,T]$.   Throughout this section, we always assume that $a(\x)\in W^{1,\infty}(\Om)$,
$f(t,\x)\in L^2(0,T;L^2(\Om))$, $g(\x)\in L^2(0,T;H^{\frac{3}{2}}(\p\Om))\cap H^{\f 3 4}(0,T;L^2(\Omega))$
and $u_0\in H^1(\Om)$.

\subsection{Measurement data in gradient form}
Suppose that the measurement data $\nabla z^\delta$ of $\nabla u(q)$ is noisy in $\Om\times I$, with a noise level $\delta$, namely
\bb\label{eq:delta:para}
\int_0^T\|\nabla u(\qd)- \nabla z^\delta\|^2_{0,\Om}dt \leq\delta^2\,,
\ee
where  $\qd$ is the true physical radiativity. We transform the inverse problem into the following
output least-squares formulation with Tikhonov regularization:
\begin{align}
\min_{q\in K}J_{\delta,\beta}(q)=\min_{q\in K}\Big(\f{1}{2}{\int_I\int_{\Omega}  |\nabla u(q)-\nabla z^\delta|^2 d\x dt}
+ \f \beta  2 \|q- q^*\|^2_{0,\Om}\Big),
\label{min:para}
\end{align}
where $\beta>0$ is the regularization parameter, $q^*\in K$ is an a priori estimate of the true
parameter $\qd$.

\begin{theorem}\label{lemma:exisp}
 There exists at least a minimizer to optimization problem \eqref{min:para}.
\end{theorem}
\begin{proof} We will omit the proof, which is not the focus of the paper and quite similar to the one of Theorem 2.1 in \cite{keung98}.
\end{proof}

We are now establishing the following Lipschitz type stability estimate for the parabolic inverse radiativity problem.

\begin{lemma}\label{lemma:instab:para}
Assume $|u(\qd)|\geq \bar c_0$ for some positive constant $\bar c_0$ in $I\times\Om$.
\begin{enumerate}

\item[\textup{(a)}] If the space dimension $d=2$, then, for any $\epsilon\in (0,1/2)$,
\begin{equation}\label{est:para2}
\|q-\qd\|_{H^{-1-\epsilon}(\Omega)}\leq C  \|u(q)-u(\qd)\|_{L^2(I;H_0^1(\Omega))}\quad \forall\, q\in K.
\end{equation}

\item[\textup{(b)}] If the space dimension $d=3$ and
\begin{equation}\label{ass:d3}
\p_t u(q^\dag)\in L^2(I;L^3(\Omega)),
\end{equation}
then \eqref{est:para2} still holds for any $\epsilon\in (0,1/2)$.

\end{enumerate}

\begin{proof}

It is easy to see from \eqref{q1p} that
$$
\p_tu(q)-\p_tu(\qd)-\nabla\cdot(a\nabla\cdot(u(q)-u(\qd)))+q(u(q)-u(\qd))=u(\qd)(\qd-q) ~~{\rm in}~~I\times\Om.
$$
 By choosing an arbitrary $\phi\in H_0^1(I,L^2(\Omega))\cap L^2(I;H_0^1(\Omega))$ and multiplying both-hands sides with it and integration over
$I\times \Omega$, we have
\begin{align*}
\left|\int_I \int_\Omega (\qd-q)u(\qd)\phi d\x dt\right|
\leq&
\left|\int_I\int_\Omega (u(q)-u(\qd))\p_t\phi d\x dt\right|
+\left|\int_I\int_\Omega a\nabla (u(q)-u(\qd)\cdot \nabla \phi d\x dt \right|\\
&+\left|\int_I\int_\Omega q (u(q)-u(\qd) \phi d\x dt \right|,
\end{align*}
which yields
\begin{align}\label{ineq:ll}
\left|\int_I \int_\Omega (\qd-q)u(\qd)\phi dx dt\right|
\leq (1+\|a\|_{L^\infty(\Omega)}+\|q\|_{L^\infty(\Omega)})\|u(q)-u(\qd)\|_{L^2(I;H_0^1(\Omega))} \notag\\
\times (\|\p_t\phi\|_{L^2(I,L^2(\Omega))}+\|\phi\|_{L^2(I;H_0^1(\Omega))}).
\end{align}
In particular, let us fix some $\varphi\in C_c^\infty(I)$ such that $\int_I \varphi(t) dt=1$ and set
\beqn\label{def:phi}
\phi_h:=hG\,\,\textup{with}\,\,  G:=\f{ \varphi}{u(\qd)}
\eqn
for $h\in H_0^{1+\epsilon}(\Omega)$.   From the hypothesis  $|u(\qd)|\geq \bar c_0$ and the fact that $u(\qd)\in L^2(I;H^2(\Omega))$ (by Lemma \ref{lem:wellt}), 
we  can infer by Leibniz's rule that
\begin{equation}\label{eq:G}
 G\in
L^2(I;H^2(\Omega))\cap H^1_0(I;L^2(\Omega)).
\end{equation}
Using Remark \ref{remark:product}  and \eqref{eq:G} , we have
\begin{align}\label{ineq:llI}
\|\phi_h\|_{L^2(I,H_0^1(\Omega))}^2
= &\int_I \|hG(t)\|_{H_0^1(\Omega)}^2 dt\leq C  \int_I \|h\|^2_{1+\epsilon,\Omega }\|G\|_{2,\Omega}^2 dt\notag\\
= &\|h\|^2_{1+\epsilon,\Omega }\|G\|^2_{L^2(I;H^2(\Omega))}.
\end{align}

For the space dimension $d=2$, we know that $h\in L^\infty(\Omega)$ and $\|h\|_{L^\infty(\Omega)}\leq C \|h\|_{1+\epsilon,\Omega}$ for all $h\in H_0^{1+\epsilon}(\Omega)$ by Sobolev embedding theorem. Therefore,  we have
\begin{align}\label{ineq:lV}
\|\p_t \phi_h\|_{L^2(I,L^2(\Omega))}^2
=\int_I\|h \p_t G \|_{L^2(\Omega)}^2 dt
\leq  \|h\|_{L^\infty(\Omega)}^2\int_I\|\p_t G \|_{L^2(\Omega)}^2 dt\leq C \|h\|_{1+\epsilon,\Omega}^2.
\end{align}

For the space dimension $d=3$,  from the hypothesis $|u(\qd)|\geq \bar c_0$ and \eqref{ass:d3}, we get 
$\p_t G\in L^2(I;L^3(\Omega))$. Therefore, 
\begin{align}\label{ineq:lV2}
\|\p_t \phi_h\|_{L^2(I,L^2(\Omega))}^2
=\int_I\|h \p_t G \|_{L^2(\Omega)}^2 dt
\leq  \|h\|_{L^6(\Omega)}^2\int_I\|\p_t G \|_{L^3(\Omega)}^2 dt\leq C \|h\|_{1+\epsilon,\Omega}^2,
\end{align}
where we have used the Sobolev embedding result $H^1(\Omega)\Subset L^6(\Omega)$.   Taking $\phi=\phi_h$ in \eqref{ineq:ll}, and using estimates
\eqref{ineq:llI}-\eqref{ineq:lV2},   we can conclude that for all $q\in K$,
\begin{align}\label{ineq:V}
\left| \int_\Omega (\qd-q)h dx \right|
\leq C\|u(q)-u(\qd)\|_{L^2(I;H_0^1(\Omega))}\|h\|_{1+\epsilon,\Omega} \quad \forall\, h\in H_0^{1+\epsilon}(\Omega),
\end{align}
which implies \eqref{est:para2}.


\end{proof}

\end{lemma}
\begin{remark}
The assumption  \eqref{ass:d3} holds provided that the source terms $f,g$  and initial value $u_0$ are smooth enough. In particular, if
$f\mid_{I\times \Omega} \in L^3(I\times \Omega)$ and
$g\mid_{I\times \Gamma}\in W^{\f{5}{3},\f 5 6}_3(I\times \p\Omega)$ and $u_0\in W^{4/3}_3(\Omega)$, then $\p_t u(q^\dag)\in L^3(I\times \Omega)$ 
(see \cite[Chapter VI]{Ladyzenskaja})  
and hence \eqref{ass:d3}  is true.  For the definition of Sobolev-Slobodeckij-type spaces $W^{\f{5}{3},\f 5 6}_3(I\times \p\Omega)$ 
and $W^{4/3}_3(\Omega)$, we can refer to \cite{Ladyzenskaja}.

\end{remark}

~

Next we introduce the following VSC: for any $q\in K$,
	\beqn\label{VSC:para}
	\f{1}{4}\|q-q^\dag\|^2_{0,\Omega}
	\le \f{1}{2}\|q-q^*\|_{0,\Omega}^2-\f{1}{2}
	\|q^\dag-q^*\|_{0,\Omega}^2+C\|u(q)-u(q^\dag)\|_{L^2(I;H_0^1(\Omega))}^\alpha,
	\eqn
and its equivalent form
	\beqn\label{innerp}
( q^\dag-q^*,q^\dag-q)_{\Omega}
\le \f{1}{4}\|q-q^\dag\|^2_{0,\Omega}
+C\|u(q)-u(q^\dag)\|_{L^2(I;H_0^1(\Omega))}^\alpha
\eqn
with some parameter $\alpha$ such that
	$$
	\begin{cases}
	\alpha=1 \quad &\text{if}\,\, \kappa>1\\
	\alpha<\f{2\kappa}{1+2\kappa}\,\textup{but it can be choosen arbitrarily close to}\,\,
\f{2\kappa}{1+\kappa}  \quad &\text{if}\,\, \kappa\in(0,\frac{1}{2})\cup (\frac{1}{2},1).\\
	\end{cases}
	$$

\begin{theorem}\label{the:vsc:para}
  Assume $|u(\qd)|\geq \bar c_0$ and  $q^\dag-q^*\in H_0^\kappa(\Omega)$ with $\kappa>0$ and $\kappa\neq 1/2$, 
  and \eqref{ass:d3} holds when the space dimension $d=3$,
then the  VSC \eqref{innerp}
 holds.

\end{theorem}

\begin{proof} As the proof is quite similar to the one of Theorem \ref{the:vsc}, we shall  confine ourselves with a sketch of proof.

If  $\qd-q^*=0 $, then we are done.  For the case when $\qd-q^*\neq 0 $ and $\kappa>1$,
we use Lemma \ref{lemma:instab:para} to ensure
 $$
 |(\qd-q^*,\qd-q)_{\Omega}|\leq \|\qd-q^*\|_{H_0^\kappa(\Omega)}\|\qd-q\|_{H^{-\kappa}(\Omega)}
 \leq C\|\qd-q^*\|_{H_0^\kappa(\Omega)}\|u(\qd)-u(q)\|_{L^2(I;H_0^1(\Omega))}.
 $$

Next, for the case when $\qd-q^*\neq 0 $ and $\kappa\in (0,\frac{1}{2})\cup (\frac{1}{2},1)$,
by constructing the same projections $\{P_\lambda\}_{\lambda>0}$ as in the proof of Theorem \ref{the:vsc}
and using the same reasoning leading to  \eqref{eq:vsc:00} and \eqref{eq:vsc:0}, we can prove that for any $\lambda>0$,
\begin{equation}\label{eq:vsc:para:01}
 |((I-P_\lambda)(\qd-q^*),\qd-q)_{\Omega}|
\leq\f{\|\qd-q\|^2}{4}+C
\f{\|\qd-q^*)\|_{D(A^{\kappa/2})}^2}{\lambda^{\kappa}},
\end{equation} 	
and any $s>1$
\beqn\label{eq:vsc:para:02}
 |(P_\lambda(\qd-q^*),\qd-q)_{\Omega}|\leq C  \lambda^{\f{s-\kappa}{2}}\|\qd-q^*\|_{D(\A^{\kappa/2})}\|u(\qd)-u(q)\|_{L^2(I;H_0^1(\Omega))}.
\eqn	
Then, the combination of \eqref{eq:vsc:para:01} and
\eqref{eq:vsc:para:02} yields
\begin{align}\label{vsc:para:final}
(\qd-q^*,\qd-q)_{\Omega}\leq& \f{\|\qd-q\|^2}{4}\\
&+CA\inf_{\lambda>0}
\left(\f{A}{\lambda^{ {\kappa}}}+\lambda^{\f{s-\kappa}{2}}\|u(\qd)-u(q)\|_{L^2(I;H_0^1(\Omega))}\right)\notag
\end{align}
with $A=\|\qd-q^*\|_{D(\A^{\kappa/2})}$, which completes the proof by balancing the parameter $\lambda$ as
in the proof of Theorem \ref{the:vsc}.
\end{proof}
With the aid of the proposed VSC \eqref{VSC:para}, we are able to prove the following convergence results,
whose proof follows the same manner of Theorem \ref{the:vsc}.

\begin{theorem}\label{the:para:con}
  Assume $|u(\qd)|\geq \bar c_0$ and  $q^\dag-q^*\in H_0^\kappa(\Omega)$ with $\kappa>0$ and $\kappa\neq 1/2$,  and \eqref{ass:d3} holds when the space dimension $d=3$.
Let  $\alpha$ be the
parameter chosen as in Theorem \ref{the:vsc:para}. Then we have the following convergence rates results,
\beqn\label{eq:convergence:para0}
\|\nabla u(\qda)-\nabla u(q^\dag)\|_{L^2(I;L^2(\Om))}=O(\delta) \quad (\delta\to 0)
\eqn
and
\beqn\label{eq:convergence:para1}
\|\qda-\qd\|_{0,\Omega}=O(\delta^{\f{\alpha}{2}}) \quad (\delta\to 0).
\eqn
under the parameter choice  $\beta=\delta^{2-\alpha}$.
\end{theorem}

\begin{proof}
Let $q^\delta_\beta$ be the minimizer of \eqref{min:para}, we have
\begin{align}\label{pqc}
\f{1}{2}\int_I\|\nabla u(\qda)-\nabla z^\delta\|_{0,\Om}^2dt+ \f \beta 2
\|\qda -q^*\|_{0,\Omega}^2\le& \f{1}{2}\int_I\|\nabla u(\qd)-\nabla z^\delta\|_{0,\Om}^2dt+\f  \beta 2
\|q^\dag-q^*\|_{0,\Omega}^2\notag\\
\le& \f 1 2 \delta^2+ \f \beta 2
\|q^\dag-q^*\|_{0,\Omega}^2,
\end{align}
which implies
\begin{align}\label{pqb}
\f{1}{2}
\|\qda-q^*\|^2_{0,\Omega}-
\f{1}{2}\|q^\dag-q^*\|_{0,\Omega}^2\le& \f{\delta^2}{2\beta}-
\f{1}{2\beta}  \int_I\|\nabla u(\qda)-\nabla z^\delta\|_{0,\Om}^2dt
\leq  \f{\delta^2}{2\beta}.
\end{align}
Then we get from  \eqref{VSC:para}, \eqref{pqb} and triangle inequality that
\begin{align}\label{pqd1}
0\le& \f{1}{2}\|\qda-q^*\|_{0,\Omega}^2-\f{1}{2}
\|\qd-q^*\|_{0,\Omega}^2+C\|u(\qda)-u(q^\dag)\|_{L^2(I;H^1_0(\Om))}^\alpha  \notag \\
\le& \f{1}{2\beta}\left(\delta^2-\int_I\|\nabla u(\qda)-\nabla z^\delta\|_{0,\Om}^2dt\right)+C\|u(\qda)-u(q^\dag)\|_{L^2(I;H^1_0(\Om))}^\alpha \nb\\
\le& \f{1}{2\beta}\left(2\delta^2-\frac{1}{2}\int_I\|\nabla u(\qda)-\nabla u(\qd)\|_{0,\Om}^2dt\right)+C\|u(\qda)-u(q^\dag)\|_{L^2(I;H^1_0(\Om))}^\alpha .
\end{align}
As $u(\qda)-u(\qd)\in L^2(0,T;H^1_0(\Om))$, then by the Poinc$\acute{a}$re's inequality, we have
\beqnx
\|u(\qda)-u(q^\dag)\|_{L^2(I;H^1_0(\Om))} \leq
C\|\nabla u(\qda)-\nabla u(q^\dag)\|_{L^2(I;L^2(\Om))},
\eqnx
which together with \eqref{pqd1} yields that
\beqn\label{eq:pkey}
\int_I\|\nabla u(\qda)-\nabla u(\qd)\|_{0,\Omega}^2dt
\le 4\delta^2  +C\beta \|\nabla u(\qda)-\nabla u(q^\dag)\|^\alpha_{L^2(I;L^2(\Om))}.
\eqn

Hence,  if $ \|\nabla u(\qda)-\nabla u(q^\dag)\|_{L^2(I;L^2(\Om))}<\delta$, then  the convergence
rate \eqref{eq:convergence:para0} holds. Otherwise, if
 $ \|\nabla u(\qda)-\nabla u(q^\dag)\|_{L^2(I;L^2(\Om))}\geq \delta$, as $\alpha\leq 1$, then one has
 $$
  \|\nabla u(\qda)-\nabla u(q^\dag)\|_{L^2(I;L^2_0(\Om))}^\alpha
  \leq   \|\nabla u(\qda)-\nabla u(q^\dag)\|_{L^2(I;L^2(\Om))} \delta^{\alpha-1}.
 $$
Taking the above inequality into \eqref{eq:pkey} and choosing $\beta=\delta^{2-\alpha}$, we get
\beqnx
\|\nabla u(\qda)-\nabla u(q^\dag)\|_{L^2(I;L^2(\Om))}^2
&\le& 4\delta^2  +C\delta^{2-\alpha}\|\nabla u(\qda)-\nabla u(q^\dag)\|_{L^2(I;L^2(\Om))}\delta^{\alpha-1}\\
&=&4\delta^2  +C\delta\|\nabla u(\qda)-\nabla u(\qd)\|_{L^2(I;L^2(\Om))}\\
&\leq&C\delta\|\nabla u(\qda)-\nabla u(\qd)\|_{L^2(I;L^2(\Om))},
\eqnx
which implies $\|\nabla u(\qda)-\nabla u(\qd)\|_{L^2(I;L^2(\Om))}=O(\delta)$. Therefore,
  \eqref{eq:convergence:para0} holds.

Further, using  \eqref{VSC:para}, \eqref{pqb} and Poinc$\acute{a}$re's inequality, we obtain
\beqnx
\f{1}{4}\|\qda-q^\dag\|^2_{0,\Omega}
&\le& \f{1}{2}\|\qda-q^*\|_{0,\Omega}^2-\f{1}{2}
\|q^\dag-q^*\|_{0,\Omega}^2+C\| u(\qda)- u(q^\dag)\|^\alpha_{L^2(I;H^1_0(\Om))}\\
&\le& \frac{\delta^2}{2\beta}+C\|\nabla u(\qda)-\nabla u(q^\dag)\|^\alpha_{L^2(I;L^2(\Om))}.
\eqnx
Then choosing $\beta=\delta^{2-\alpha}$ and using \eqref{eq:convergence:para0}, we have
\beqnx
\f{1}{4}\|\qda-q^\dag\|^2_{0,\Omega}
\le \frac{\delta^2}{2\delta^{2-\alpha}}+C\delta^\alpha\leq C\delta^\alpha,
\eqnx
which verifies \eqref{eq:convergence:para1}.

\end{proof}

 Similar to Corollary \ref{coro:ellip}, we can obtain the following result.

\begin{corollary}
Under the hypothesises and settings of Theorem \ref{the:para:con} and assume that $2\leq p<+\infty$, we then have the convergence rate
\beqn\label{eq:convergence12}
\|\qda-\qd\|_{L^p(\Omega)}=O(\delta^{\f{\alpha}{p}}) 
\eqn
under the parameter choice  $\beta=\delta^{2-\alpha}$.

\end{corollary}

\subsection{Measurement data in $L^2$-norm}

In this subsection, we assume that the noisy data $z^\delta$
 satisfies
 \beqn\label{eq:delta:para:l2}
\int_I\int_{\Omega}  | u(\qd)- z^\delta|^2 d\x dt \leq \delta^2
\eqn
and $\qda$  is the minimizer of the following
output least-squares formulation with Tikhonov regularization:
\begin{align}
\min_{q\in K}J_{\delta,\beta}(q)=\min_{q\in K}\Big(\f{1}{2}{\int_I\int_{\Omega}  | u(q)- z^\delta|^2 d\x dt}
+ \f \beta  2 \|q- q^*\|^2_{0,\Om}\Big),
\label{min:para}
\end{align}
where $\beta>0$ is the regularization parameter, $q^*\in K$ is an a priori estimate of the true
parameter $\qd$. Our goal is to study the convergence rate of the regularized solution $\qda$.  To this end, we
follow the same procedure used in Subsection 4.1, and first establish the following H\"{o}lder type (conditional)
estimate of the parabolic inverse radiativity problem.

 \begin{lemma}\label{lemma:instab:para:l2}
Assume $|u(\qd)|\geq \bar c_0$ for some positive constant $\bar c_0$ in $I\times\Om$.
\begin{enumerate}

\item[\textup{(a)}] If the space dimension $d=2$, then, for any $\epsilon\in (0,1/2)$,
\begin{equation}\label{est:para2:l2}
\|q-\qd\|_{H^{-1-\epsilon}(\Omega)}\leq C  \|u(q)-u(\qd)\|_{L^2(I;L^2(\Omega))}^{\f{1}{2}}\quad \forall\, q\in K.
\end{equation}

\item[\textup{(b)}] If the space dimension $d=3$ and \eqref{ass:d3} is fulfilled,
then \eqref{est:para2:l2} still holds for any $\epsilon\in (0,1/2)$.

\end{enumerate}
\end{lemma}
\begin{proof} From Cauchy's inequality and \eqref{Sob:interp} it follows that
\beqn
\|u\|_{L^2(I;H^1(\Omega))}\leq C\|u\|_{L^2(I;H^2(\Omega))}^{\f{1}{2}}\|u\|_{L^2(I;L^2(\Omega))}^{\f{1}{2}}\quad\forall\,u\in L^2(I;H^2(\Omega))
\eqn
Thus, we only need to show that
\beqn\label{para:l2:g}
\|u(q)-u(\qd)\|_{L^2(I;H^2(\Omega))}\leq C.
\eqn
  It is easy to see from \eqref{q1p} that $w=u(q)-u(\qd)$ satisfies
\begin{equation}
\left\{ \begin{array}{rlclc}
\p_t w-\nabla\cdot(a(\x)\nabla w)+q(\x)w&=&u(\qd)(\qd-q) &\m{in} &\Om\times (0,T], \\
 w(\x,0)&=&0 &\m{in}  &\Om,\\
 w(\x,t)&=&0 &\m{on}  &\p\Om\times (0,T],
\end{array}
\right. \label{eq:para:l2:0}
\end{equation}
Then by making use of Lemma \ref{lem:wellt}, we get
\beqnx
\|u(q)-u(\qd)\|_{L^2(I;H^2(\Omega))}\le C\|u(\qd)(\qd-q)\|_{L^2(0,T;L^2(\Omega))}
\leq 2C\overline{q}\|u(\qd)\|_{L^2(0,T;L^2(\Omega))},
\eqnx
which infers that \eqref{para:l2:g} is valid.

\end{proof}

Using Lemma \ref{lemma:instab:para:l2} and the arguments leading to Theorem \ref{the:vsc:para:l2} and Theorem \ref{the:para:con:l2}, we can prove the analogues of Theorem \ref{the:vsc:para:l2} and Theorem \ref{the:para:con:l2}:
\begin{theorem}\label{the:vsc:para:l2}
  Assume $|u(\qd)|\geq \bar c_0$ and  $q^\dag-q^*\in H_0^\kappa(\Omega)$ with $\kappa>0$ and $\kappa\neq 1/2$,  and \eqref{ass:d3} holds when the space dimension $d=3$.
Then the following VSC: for any $q\in K$,
	\beqn\label{VSC:para}
	\f{1}{4}\|q-q^\dag\|^2_{0,\Omega}
	\le \f{1}{2}\|q-q^*\|_{0,\Omega}^2-\f{1}{2}
	\|q^\dag-q^*\|_{0,\Omega}^2+C\|u(q)-u(q^\dag)\|_{L^2(I;L^2(\Omega))}^\alpha,
	\eqn
 holds with some $C>0$ and parameter $\alpha>0$ such that
	$$
	\begin{cases}
	\alpha=1/2 \quad &\text{if}\,\, \kappa>1\\
	\alpha<\f{\kappa}{1+2\kappa}\,\textup{but it can be choosen arbitrarily close to}\,\,\f{2\kappa}{1+\kappa}  \quad &\text{if}\,\, \kappa\in(0,1].\\
	\end{cases}
	$$

\end{theorem}

\begin{theorem}\label{the:para:con:l2}
  Assume $|u(\qd)|\geq \bar c_0$ and  $q^\dag-q^*\in H_0^\kappa(\Omega)$ with $\kappa>0$ and $\kappa\neq 1/2$,   and \eqref{ass:d3} holds when the space dimension $d=3$.
Let  $\alpha$ be the
parameter chosen as in Theorem \ref{the:vsc:para}. Then we have the following convergence rates results,
\beqn\label{eq:convergence:para0}
\| u(\qda)-u(q^\dag)\|_{L^2(I;L^2(\Om))}=O(\delta) \quad (\delta\to 0)
\eqn
and
\beqn\label{eq:convergence:para1}
\|\qda-\qd\|_{0,\Omega}=O(\delta^{\f{\alpha}{2}}) \quad (\delta\to 0).
\eqn
under the parameter choice  $\beta=\delta^{2-\alpha}$.
\end{theorem}

\section{Concluding remarks }\label{sec:conclu}
We have justified in this work the conditional stability estimates of the elliptic and parabolic
inverse radiativity  problems. We have also proposed some new variational source conditions,
which are rigorously verified under the conditional stabilities in general dimensional spaces.
With these variational source conditions,
the reasonable convergence rates results are achieved. 

In the future work, we shall consider some elliptic and parabolic
inverse radiativity and conductivity  problems with measurable data in some subdomain of $\Om$.
We hope to propose some variational source conditions, which can be verified rigorously, and
derive some corresponding convergence rates results.

\end{document}